\documentclass[11pt]{amsart}

\usepackage{amsmath}
\usepackage{amssymb}
\usepackage{amsthm}
\usepackage{fancyhdr}
\usepackage{enumerate}
\usepackage{eucal}
\usepackage{fixmath}
\usepackage[enableskew]{youngtab}
\usepackage[square]{natbib}
\usepackage{tikz}

\newtheorem{thm}{Theorem}
\newtheorem{lemma}[thm]{Lemma}
\newtheorem{prop}[thm]{Proposition}
\newtheorem{conj}{Conjecture}

\newcommand{\wt}[1] {{\widetilde{#1}}}
\newcommand{\Z} {\mathbb{Z}}
\newcommand{\N} {\mathbb{N}}
\newcommand{\Inv} {\mathrm{Inv}}
\newcommand{\affinv}{\Inv_{\wt{S}_n}}

\setcitestyle{numbers}

\author{Andrew Crites}
\thanks{Andrew Crites acknowledges support from grant DMS-0800978 from the National Science Foundation.}
\title{Enumerating Pattern Avoidance for Affine Permutations}
\address{Department of Mathematics, University of Washington, Box 354350, Seattle, Washington, 98195-4350, acrites@math.washington.edu}

\begin{document}
\begin{abstract}
In this paper we study pattern avoidance for affine permutations.
In particular, we show that for a given pattern $p$, there are only finitely many affine permutations in $\widetilde{S}_n$
that avoid $p$ if and only if $p$ avoids the pattern 321.
We then count the number of affine permutations that avoid a given pattern $p$ for each $p$ in $S_3$,
as well as give some conjectures for the patterns in $S_4$.
\end{abstract}

\maketitle

\section{Introduction}
\label{sec:in}
Given a property $Q$, it is a natural question to ask if there is a simple characterization of all permutations with property $Q$.
For example, in \cite{LakSan} the permutations corresponding to smooth Schubert varieties
are exactly the permutations that avoid the two patterns 3412 and 4231.
In \cite{Tenner} it was shown that the permutations with Boolean order ideals
are exactly the ones that avoid the two patterns 321 and 3412.
A searchable database listing which classes of permutations avoid certain patterns can be found at \cite{TennerDatabase}. 

Since we know pattern avoidance can be used to describe useful classes of permutations,
we might ask if we can enumerate the permutations avoiding a given pattern or set of patterns.
For example, in \cite{MarcusTardos} it was shown that if $S_n(p)$ is the number of permutations in the symmetric group, $S_n$, that avoid the pattern $p$,
then there is some constant $c$ such that $S_n(p)\le c^n$.
Thus the rate of growth of pattern avoiding permutations is bounded.
This result was known as the Stanley-Wilf conjecture, now called the Marcus-Tardos Theorem.

We can express elements of the affine symmetric group, $\wt{S}_n$, as an infinite sequence of integers,
and it is still natural to ask if there exists a subsequence with a given relative order.
Thus we can extend the notion of pattern avoidance to these affine permutations and
we can try to count how many $\omega\in\wt{S}_n$ avoid a given pattern.

For $p\in S_m$, let \begin{equation}f_n^p=\#\left\{\omega\in\wt{S}_n:\omega\textrm{ avoids }p\right\}\label{eqn:genfuncoeff}\end{equation}
and consider the generating function
\begin{equation}f^p(t)=\sum_{n=2}^\infty f_n^pt^n.\label{eqn:genfun}\end{equation}
For a given pattern $p$ there could be infinitely many $\omega\in\wt{S}_n$ that avoid $p$.
In this case, the generating function in \eqref{eqn:genfun} is not even defined.
As a first step towards understanding $f^p(t)$, we will prove the following theorem.
\begin{thm}
\label{mainthm}
Let $p\in S_m$.
For any $n\ge2$ there exist only finitely many $\omega\in\wt{S}_n$ that avoid $p$ if and only if $p$ avoids the pattern 321.
\end{thm}

It is worth noting that 321-avoiding permutations and 321-avoiding affine permutations
appear as an interesting class of permutations in their own right.
In \cite[Theorem 2.1]{BJS} it was shown that a permutation is fully commutative if and only if it is 321-avoiding.
This means that every reduced expression may be obtained from any other reduced expression using only relations of the form $s_is_j=s_js_i$.
Moreover, a proof that this result can be extended to affine permutations as well appears in \cite[Theorem 2.7]{Affine321}.
For a detailed discussion of fully commutative elements in other Coxeter groups, see \cite{Stembridge}.

Even in the case where there might be infinitely many $\omega\in\wt{S}_n$ that avoid a pattern $p$,
we can always construct the following generating function.
Let \begin{equation}g^p_{m,n}=\#\left\{\omega\in\wt{S}_n:\omega\textrm{ avoids }p\textrm{ and }\ell(\omega)=m\right\}.\label{eqn:genfuncoeff2}\end{equation}
Then set \begin{equation}g^p(x,y)=\sum_{n=2}^\infty\sum_{m=0}^\infty g^p_{m,n}x^my^n.\label{eqn:genfun2}\end{equation}
Since there are only finitely many elements in $\wt{S}_n$ of a given length, we always have $g^p_{m,n}<\infty$.
The generating function $g^{321}(x,y)$ is computed in \cite[Theorem 3.2]{JonesHanusa}.

The outline of this paper is as follows.
In Section \ref{sec:back} we will review the definition of the affine symmetric group and list several of its useful properties.
In Section \ref{sec:mainproof} we will prove Theorem \ref{mainthm},
which will follow immediately from combining Propositions \ref{if} and \ref{onlyif}.
In Section \ref{sec:calcs} we will compute $f^p(t)$ for all of the patterns in $S_3$.
Finally, in Section \ref{sec:calcs2} we will give some basic results and conjectures for $f^p(t)$ for the patterns in $S_4$.

\section{Background}
\label{sec:back}
Let $\wt{S}_n$ denote of the set of all bijections $\omega:\Z\to\Z$ with
$\omega(i+n)=\omega(i)+n$ for all $i\in\Z$ and
\begin{equation}\sum_{i=1}^n\omega(i)=\binom{n+1}{2}.\label{eqn:affinesum}\end{equation}
$\wt{S}_n$ is called the \emph{affine symmetric group}, and the elements of $\wt{S}_n$ are called \emph{affine permutations}.
This definition of affine permutations first appeared in \cite[\S3.6]{Lusztig83} and was then developed in \cite{Shi}.
Note that $\wt{S}_n$ also occurs as the affine Weyl group of type $A_{n-1}$.

We can view an affine permutation in its one-line notation as the infinite string of integers
\begin{displaymath}\cdots\omega_{-1}\omega_0\omega_1\cdots\omega_n\omega_{n+1}\cdots,\end{displaymath}
where, for simplicity of notation, we write $\omega_i=\omega(i)$.
An affine permutation is completely determined by its action on $[n]:=\{1,\dots,n\}$.
Thus we only need to record the base window $[\omega_1,\dots,\omega_n]$ to capture all of the information about $\omega$.
Sometimes however, it will be useful to write down a larger section of the one-line notation.

Given $i\not\equiv j\mod{n}$, let $t_{ij}$ denote the affine transposition that interchanges $i+mn$ and $j+mn$ for all $m\in\Z$
and leaves all $k$ not congruent to $i$ or $j$ fixed.
Since $t_{ij}=t_{i+n,j+n}$ in $\wt{S}_n$, it suffices to assume $1\le i\le n$ and $i<j$.
Note that if we restrict to the affine permutations with $\{\omega_1,\dots,\omega_n\}=[n]$,
then we get a subgroup of $\wt{S}_n$ isomorphic to $S_n$, the group of permutations of $[n]$.
Hence if $1\le i<j\le n$, the above notion of transposition is the same as for the symmetric group.

Given a permutation $p\in S_k$ and an affine permutation $\omega\in\wt{S}_n$,
we say that \emph{$\omega$ avoids the pattern $p$} if there is no subsequence of integers $i_1<\cdots<i_k$ such that
the subword $\omega_{i_1}\cdots\omega_{i_k}$ of $\omega$ has the same relative order as the elements of $p$.
Otherwise, we say that $\omega$ \emph{contains} $p$.
For example, if $\omega=[8,1,3,5,4,0]\in\wt{S}_6$, then 8,1,5,0 is an occurrence of the pattern 4231 in $\omega$.
However, $\omega$ avoids the pattern 3412.
A pattern can also come from terms outside of the base window $[\omega_1,\dots,\omega_n]$.
In the previous example, $\omega$ also has 2,8,6 as an occurrence of the pattern 132.
Choosing a subword $\omega_{i_1}\cdots\omega_{i_k}$ with the same relative order as $p$ will be referred to as \emph{placing} $p$ in $\omega$.

\subsection{Coxeter Groups}
\label{sec:CoxGroups}
For a general reference on the basics of Coxeter groups, see \cite{BB:05} or \cite{HumphreysCoxGrp}.
Let $S=\{s_1,\dots,s_n\}$ be a finite set, and let $F$ denote the free group consisting of all words of finite length whose letters come from $S$.
Here the group operation is concatenation of words, so that the empty word is the identity element.
Let $M=\left(m_{ij}\right)_{i,j=1}^n$ be any symmetric $n\times n$ matrix whose entries come from $\Z_{>0}\cup\{\infty\}$ with 1's on the diagonal.
Then let $N$ be the normal subgroup of $F$ generated by the relations
\begin{displaymath}R=\left\{(s_is_j)^{m_{ij}}=1\right\}_{i,j=1}^n.\end{displaymath}
If $m_{ij}=\infty$, then there is no relationship between $s_i$ and $s_j$.
The Coxeter group corresponding to $S$ and $M$ is the quotient group $W=F/N$.

Any $w\in W$ can be written as a product of elements from $S$ in infinitely many ways.
Every such word will be called an \emph{expression} for $w$.
Any expression of minimal length will be called a \emph{reduced expression}, and the number of letters in such an
expression will be denoted $\ell(w)$, the \emph{length} of $w$.
Call any element of $S$ a \emph{simple reflection} and any element conjugate to a simple reflection, a \emph{reflection}.

We graphically encode the relations in a Coxeter group via its \emph{Coxeter graph}.
This is the labeled graph whose vertices are the elements of $S$.
We place an edge between two vertices $s_i$ and $s_j$ if $m_{ij}>2$ and we label the edge $m_{ij}$ whenever $m_{ij}>3$.
The Coxeter graphs of all the finite Coxeter groups have already been classified.
See, for example, \cite[\S2]{HumphreysCoxGrp}.

In \cite[\S8.3]{BB:05} it was shown that $\wt{S}_n$ is the Coxeter group with generating set $S=\{s_0,s_1,\dots,s_{n-1}\}$, and relations
\begin{displaymath}
R=\begin{cases}s_i^2=1,\\\left(s_is_j\right)^2=1,&\textrm{if }|i-j|\ge2,\\\left(s_is_{i+1}\right)^3=1,&\textrm{for }0\le i\le n-1,\end{cases}
\end{displaymath}
where all of the subscripts are taken mod $n$.
Thus the Coxeter graph for $\wt{S}_n$ is an $n$-cycle, where every edge is unlabeled.

\begin{figure}[h]
\begin{center}
\begin{tikzpicture}[style=thick]
\draw[fill=black] {
(4,2) circle (2pt) -- (1,0)
(1,0) circle (2pt) -- (3,0)
(3,0) circle (2pt)
(5,0) circle (2pt) -- (7,0)
(7,0) circle (2pt) -- (4,2)};
\draw {
(4,2) node[above] {$s_0$}
(1,0) node[below] {$s_1$}
(3,0) node[below] {$s_2$}
(4,0) node {$\cdots$}
(5,0) node[below] {$s_{n-2}$}
(7,0) node[below] {$s_{n-1}$}};
\end{tikzpicture}
\caption{Coxeter graph for $\wt{S}_n$.}
\end{center}
\end{figure}
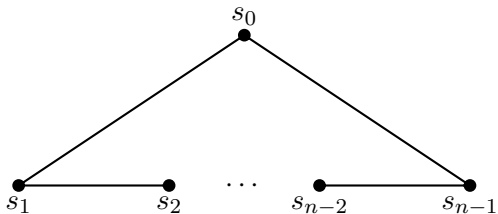

If $J\subsetneq S$ is a proper subset of $S$, then we call the subgroup of $W$ generated by just the elements of $J$ a \emph{parabolic subgroup}.
Denote this subgroup by $W_J$.
In the case of the affine symmetric group we have the following characterization of parabolic subgroups.
\begin{prop}
\label{parabolic}
Let $J=S\backslash\{s_i\}$.
Then $\omega\in\wt{S}_n$ is in the parabolic subgroup $(\wt{S}_n)_J$ if and only if there exists some integer $i\le j\le i+n-1$ such that
$\omega_j\le\omega_k<\omega_j+n$ for all $i\le k\le i+n-1$.
\end{prop}

\begin{proof}
In \cite[Proposition 8.3.4]{BB:05} it is shown that  $(\wt{S}_n)_J=\mathrm{Stab}([i,i+n-1])$.
The result then follows from the definition of the stabilizing set.
\end{proof}

\subsection{Length Function for $\wt{S}_n$}
\label{sec:length}
For $\omega\in\wt{S}_n$, let $\ell(\omega)$ denote the length of $\omega$ when $\wt{S}_n$ is viewed as a Coxeter group.
Recall that for a non-affine permutation $\pi\in S_n$ we can define an \emph{inversion} as a pair $(i,j)$ such that $i<j$ and $\pi_i>\pi_j$.
For an affine permutation, if $\omega_i>\omega_j$ for some $i<j$, then we also have $\omega_{i+kn}>\omega_{j+kn}$ for all $k\in\Z$.
Hence any affine permutation with a single inversion has infinitely many inversions.
Thus we standardize each inversion as follows.
Define an \emph{affine inversion} as a pair $(i,j)$ such that $1\le i \le n$, $i<j$, and $\omega_i>\omega_j$.
If we let $\affinv(\omega)$ denote the set of all affine inversions in $\omega$, then $\ell(\omega)=\#\affinv(\omega)$, \cite[Proposition 8.3.1]{BB:05}.

We also have the following characterization of the length of an affine permutation, which will be useful later.
\begin{thm}\textup{\cite[Lemma 4.2.2]{Shi}}
\label{length}
Let $\omega\in\wt{S}_n$.
Then \begin{equation}\ell(\omega)=\sum_{1\le i<j\le n}\left|\left\lfloor\frac{\omega_j-\omega_i}{n}\right\rfloor\right|
=\mathrm{inv}(\omega_1,\dots,\omega_n)+\sum_{1\le i<j\le n}\left\lfloor\frac{|\omega_j-\omega_i|}{n}\right\rfloor,\label{eqn:length}\end{equation}
where $\mathrm{inv}(\omega_1,\dots,\omega_n)=\#\{1\le i<j\le n:\omega_i>\omega_j\}$.
\end{thm}

For $1\le i\le n$ define $\Inv_i(\omega)=\#\{j\in\N:i<j,\omega_i>\omega_j\}$.
Now let $\Inv(\omega)=(\Inv_1(\omega),\dots,\Inv_n(\omega))$, which will be called the \emph{affine inversion table} of $\omega$.
In \cite[Theorem 4.6]{BB:96} it was shown that there is a bijection between $\wt{S}_n$ and elements of $\Z_{\ge0}^n$ containing at least one zero entry.

\section{Proof of Theorem \ref{mainthm}}
\label{sec:mainproof}
We start with the proof of one direction of Theorem \ref{mainthm}.
Proposition \ref{onlyif} will complete the proof.
\begin{prop}
\label{if}
If $p\in S_m$ contains the pattern 321, then there are infinitely many $\omega\in\wt{S}_n$ that avoid $p$.
\end{prop}

\begin{proof}
For $k\in\N$, let $\omega^{(k)}\in\wt{S}_n$ be the affine permutation whose reduced expression, when read right to left, is obtained as follows.
Starting at $s_0$, proceed clockwise around the Coxeter diagram $k(n-1)$ steps, appending each vertex as you go.
The base window of the one-line notation of these elements has the form \begin{displaymath}\omega^{(k)}=[1-k,2-k,\dots,n-1-k,n+k(n-1)].\end{displaymath}
Note these elements correspond with the spiral varieties in the affine Grassmannian from \cite{BilleyMitchell}.

As an example, in $\wt{S}_4$ we have the following:
\begin{displaymath}\begin{array}{rcccl}
s_2s_1s_0&=&\omega^{(1)}&=&[0,1,2,7]\\
s_1s_0s_3s_2s_1s_0&=&\omega^{(2)}&=&[-1,0,1,10]\\
s_0s_3s_2s_1s_0s_3s_2s_1s_0&=&\omega^{(3)}&=&[-2,-1,0,13].
\end{array}\end{displaymath}

The infinite string in the one-line notation of $\omega^{(k)}$ is a shuffle of two increasing sequences.
Hence every $\omega^{(k)}$ avoids the pattern 321.
Thus there are infinitely many permutations in $\wt{S}_n$ avoiding the pattern 321, and hence avoiding any pattern $p$ containing 321.
\end{proof}

Call a permutation $p\in S_m$ \emph{decomposable} if $p$ is contained in a proper parabolic subgroup of $S_m$.
In other words, there exists some $1\le j\le m-1$ such that $\{p_1,\dots,p_j\}=[j]$.
Then we have the following lemma.

\begin{lemma}
\label{decomplemma}
Let $p\in S_m$ be decomposable.
If $p$ avoids the pattern 321, then there exists some constant $L$ such that if $\ell(\omega)>L$, then $\omega$ contains the pattern $p$.
Hence there are only finitely many $\omega\in\wt{S}_n$ that avoid $p$.
\end{lemma}

\begin{proof}
Our proof will use induction on $m$.
If $m=1$, the result is clear, since no affine permutations can avoid $p$.
So suppose that $m>1$ and that for any $k<m$ and $q\in S_k$
there exists some constant $L_q$ such that if $\ell(\omega)>L_q$, then $\omega$ contains $q$.

Since $p$ is decomposable, there exists an index $j$ such that $\{p_1,\dots,p_j\}=[j]$.
Note in this case, we also have $\{p_{j+1},\dots,p_m\}=\{j+1,\dots,m\}$.
So we can view $q=p_1\cdots p_j$ as an element of $S_j$ and $r=p_{j+1}\cdots p_m$ as an element of $S_{m-j}$.
By our induction hypothesis, there exists a constant $L_q$ such that if $\ell(\omega)>L_q$, then there is an occurrence of $q$ in $\omega$.
Also, there exists a constant $L_r$ such that if $\ell(\omega)>L_r$, then there is an occurrence of $r$ in $\omega$.

Now let $L=\max\{L_q,L_r\}$ and suppose $\ell(\omega)>L$.
Then $\omega$ contains both patterns $q$ and $r$.
By the periodic property of $\omega$, we can translate the occurrence of $r$ to the right some multiple of $n$
until it lies entirely to the right of the occurrence of $q$ and the smallest entry of $r$ is bigger than the largest entry of $q$.
This will now give an occurrence of $p$ in $\omega$.
More specifically, let $\omega_{a_1}\cdots\omega_{a_j}$ be the occurrence of $q$ in $\omega$,
and let $\omega_{b_1}\cdots\omega_{b_{m-j}}$ be the occurrence of $r$ in $\omega$.
Then there exists some $k\in\N$ such that $b_1+kn>a_j$ and $\omega_{a_t}<\omega_{b_{s+kn}}$ for every $1\le t\le j$ and $1\le s\le m-j$.
We then have $\omega_{a_1}\cdots\omega_{a_j}\omega_{b_1+kn}\cdots\omega_{b_{m-j}+kn}$ is an occurrence of $p$ in $\omega$.

Hence if $\ell(\omega)>L$, then $\omega$ must contain $p$.
Thus there can only be finitely many $\omega\in\wt{S}_n$ that avoid $p$ since $\#\{\omega\in\wt{S}_n:\ell(\omega)\le L\}$ is finite.
\end{proof}

We now want to prove a similar statement for a general 321-avoiding pattern.
\begin{prop}
\label{onlyif}
If $p\in S_m$ avoids the pattern 321, then there are only finitely many $\omega\in\wt{S}_n$ that avoid $p$.
\end{prop}

\begin{proof}
Suppose $p$ avoids the pattern 321.
By Lemma \ref{decomplemma} we may also assume $p$ is indecomposable.
Let $a=a_1\cdots a_\ell$ be the subsequence of $p$ consisting of all $p_j$ such that $p_i<p_j$ for all $i<j$.
Here $a$ is just the sequence of left-to-right maxima.
Let $b$ be the subsequence of $p$ consisting of all $p_i$ not in $a$.
By its construction, $a$ must be increasing.
Furthermore, since $p$ avoids the pattern 321, $b$ must also be increasing.
To see this, note that if there is some $p_s,p_t$ in $b$ with $s<t$ and $p_s>p_t$, then there is some $r<s$ with $p_r>p_s$, since $p_s$ is not in $a$.
But then $p_rp_sp_t$ forms a 321 pattern in $p$.

Let $\omega\in\wt{S}_n$ and suppose that for some $1\le\alpha<\beta\le n$, we have
\begin{displaymath}\left\lfloor\frac{|\omega_\beta-\omega_\alpha|}{n}\right\rfloor>m^{\ell+1}+1.\end{displaymath}
If $\omega_\alpha<\omega_\beta$, set $\omega_\alpha^\prime=\omega_\beta$ and $\omega_\beta^\prime=\omega_\alpha+n$.
Then we will have $\omega_\alpha^\prime>\omega_\beta^\prime$ and
\begin{displaymath}\left\lfloor\frac{|\omega_\beta^\prime-\omega_\alpha^\prime|}{n}\right\rfloor>m^{\ell+1}.\end{displaymath}
So in what follows we will assume $\omega_\alpha>\omega_\beta$ and
\begin{equation}\left\lfloor\frac{|\omega_\beta-\omega_\alpha|}{n}\right\rfloor>m^{\ell+1}.\label{eqn:diffbound}\end{equation}

We can now construct the occurrence of $p$ in $\omega$.
Our iterative algorithm will complete in $\ell$ steps, where $\ell$ is the length of the subsequence $a$ described above.
We will be using translates $\omega_{\alpha+kn}$ to place the terms of $p$ in the $a$ sequence
and translates $\omega_{\beta+kn}$ to place the terms of $p$ in the $b$ sequence.

Since $p$ is indecomposable, $a_1\not=1$.
Hence there is some $t$ such that $b_t=a_1-1$.
Suppose $b_t=p_i$.
Let $s$ be the largest index such that $a_s$ lies to the left of $b_t$ in $p$.
Note that $1<s<m$ or else $p$ is decomposable.
Let $y$ be the largest integer such that $\omega_{\beta+yn}<\omega_\alpha$ and let $z=\left\lfloor\tfrac{y}{s}\right\rfloor$.
Since $\omega_\alpha-\omega_\beta>nm^{\ell+1}$, we have $y>m^{\ell+1}$ and hence $z>m^\ell$.
For each $1\le k\le s$, use $\omega_{\alpha+(k-1)zn}$ to place $a_k$ in $\omega$.
Then if $\omega_u$ corresponds to $a_k$ and $\omega_v$ corresponds to $a_{k+1}$, we will have
\begin{equation}\left|\omega_u-\omega_v\right|=\left|u-v\right|=nz>nm^\ell.\label{eqn:spacing1}\end{equation}
Finally, use translates of $\omega_\beta$ to place $b_1,\dots,b_t$ in $\omega$ in such a way that $b_t$ is placed at $\omega_{\beta+yn}$
and for any $1\le x<t$, if $b_x$ lies between $a_k$ and $a_{k+1}$ in $p$,
then $b_x$ is placed at a translate of $\omega_\beta$ between $\omega_{\alpha+(k-1)zn}$ and $\omega_{\alpha+kzn}$.
By \eqref{eqn:spacing1} there are at least $m^\ell$ translates of $\omega_\beta$ in this interval, so there is enough space
to place all of the $b_x$'s that lie between $a_k$ and $a_{k+1}$ using translates of $\omega_\beta$.
Thus after the first iteration we have placed $p_1\cdots p_i$ in $\omega$.

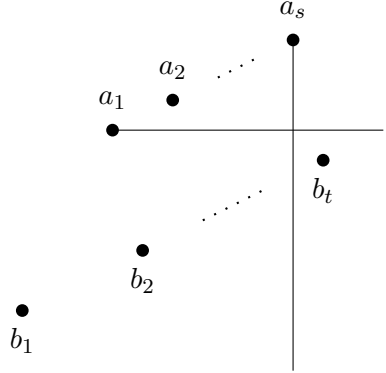
\begin{figure}
\begin{center}
\begin{tikzpicture}
\pgftransformscale{0.4}

\coordinate (a1) at (3,6);
\coordinate (a2) at (5,7);
\coordinate (as) at (9,9);
\coordinate (b1) at (0,0);
\coordinate (b2) at (4,2);
\coordinate (bt) at (10,5);

\foreach \point in {a1,a2,as,b1,b2,bt}
	{\fill [black] (\point) circle (6pt);}
	
\draw (a1) -- +(9,0);
\draw (as) -- +(0,-11);

\draw (3,7) node{$a_1$};
\draw (5,8) node{$a_2$};
\draw (9,10) node{$a_s$};
\draw (0,-1) node{$b_1$};
\draw (4,1) node{$b_2$};
\draw (10,4) node{$b_t$};

\draw[loosely dotted,thick] (6.5,7.75) -- +(1.5,0.75);
\draw[loosely dotted,thick] (6,3) -- +(2,1);

\end{tikzpicture}
\end{center}
\caption{First place all values of $p$ to the left of $b_t$.}
\end{figure}

Now suppose we have placed every term in the $a$ sequence up to $a_r$ for some $1<r<\ell$.
If we have placed $a_r$, then we have also placed some additional terms from the $b$ sequence.
Again, fix $t$ so that $b_t$ is the largest element in $p$ to the right of $a_r$ satisfying $b_t<a_r$.
We may assume such a $b_t$ exists, or else $p$ is decomposable.
If $b_t=p_i$, then we have actually placed $p_1\cdots p_i$.
Moreover, suppose that the terms from the $a$ sequence among $p_1\cdots p_i$ have been placed so that if $\omega_u$ corresponds to $a_k$
and $\omega_v$ corresponds to $a_{k+1}$ for some $1\le k\le r$, then
\begin{equation}\left|\omega_u-\omega_v\right|=\left|u-v\right|>nm^{\ell-r+1}.\end{equation}
Note we must have also already placed $a_{r+1}$, or else $a_{r+1}=p_{i+1}$ and hence $p$ is decomposable.

We will now show how to place all terms in $p$ from the $b$ sequence whose values are between $a_r$ and $a_{r+1}$,
thus completing the $(r+1)^{\text{st}}$ step of our algorithm.
Note that in the process of placing these terms, we will also possibly be placing some additional terms from the $a$ sequence.
Let $\omega_u$ correspond to $a_r$ and $\omega_v$ correspond to $a_{r+1}$.
Then we have at least $m^{\ell-r+1}$ translates of $\omega_\alpha$ and $\omega_\beta$ falling between $\omega_u$ and $\omega_v$.
So if $p_j$ is the largest entry of $p$ to the left of $a_{r+1}$ satisfying $p_j<a_{r+1}$, as in the first step of our algorithm,
we may place $p_{i+1},\dots,p_j$ in such a way that any of the terms corresponding to the subsequence $a$ are placed at least $m^{\ell-r}$ translates apart.

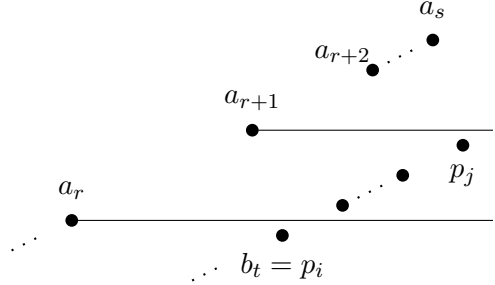
\begin{figure}
\begin{center}
\begin{tikzpicture}
\pgftransformscale{0.4}

\coordinate (ar) at (2,2);
\coordinate (ar1) at (8,5);
\coordinate (ar2) at (12,7);
\coordinate (as) at (14,8);
\coordinate (pi) at (9,1.5);
\coordinate (b1) at (11,2.5);
\coordinate (b2) at (13,3.5);
\coordinate (pj) at (15,4.5);

\foreach \point in {ar,ar1,ar2,as,pi,b1,b2,pj}
	{\fill [black] (\point) circle (6pt);}
	
\draw (ar) -- +(14,0);
\draw (ar1) -- +(8,0);

\draw (2,3) node{$a_r$};
\draw (8,6) node{$a_{r+1}$};
\draw (11,7.5) node{$a_{r+2}$};
\draw (14,9) node{$a_s$};
\draw (9,0.5) node{$b_t=p_i$};
\draw (15,3.5) node{$p_j$};

\draw[loosely dotted,thick] (0,1) -- +(1,0.5);
\draw[loosely dotted,thick] (12.5,7.25) -- +(1,0.5);
\draw[loosely dotted,thick] (6,0) -- +(1,0.5);
\draw[loosely dotted,thick] (11.5,2.75) -- +(1,0.5);

\end{tikzpicture}
\end{center}
\caption{The $(r+1)^\text{st}$ iteration will place all elements of $p$ between $p_{i+1}$ and $p_j$.}
\end{figure}

Iterating this algorithm $\ell$ times will place all of $p$ in $\omega$.
Hence if $\omega$ is to avoid $p$, then we must have
\begin{displaymath}\left\lfloor\frac{|\omega_\beta-\omega_\alpha|}{n}\right\rfloor\le m^{\ell+1}+1\text{ for all }1\le\alpha<\beta\le n.\end{displaymath}
Since $\mathrm{inv}(\omega_1,\dots,\omega_n)\le\binom{n}{2}$, we conclude by \eqref{eqn:length} that
\begin{equation}\ell(\omega)\le\binom{n}{2}+\left(m^{\ell+1}+1\right)\binom{n}{2}=\left(m^{\ell+1}+2\right)\binom{n}{2}.\label{eqn:lengthbound}\end{equation}
For any $k$, the set of all affine permutations in $\wt{S}_n$ of length at most $k$ is finite.
Hence there can be only finitely many elements in $\wt{S}_n$ that avoid $p$.
\end{proof}

Note that in general, the length bound $\ell(\omega)\le(m^{\ell+1}+2)\binom{n}{2}$ is much larger than needed.
For the proof of Theorem \ref{mainthm} though, any upper bound on $\ell(\omega)$ will suffice.
Given a specific pattern $p$, we can tighten the bounds in the above algorithm,
and thus obtain better upper bounds on the maximal length for pattern avoidance.

For example, let $p=3412\in S_4$.
By \eqref{eqn:lengthbound}, if $\omega\in\wt{S}_n$ avoids $p$, then $\ell(\omega)\le66\binom{n}{2}$.
Here the algorithm is completed on the first iteration and we can actually prove a tighter bound $\ell(\omega)\le3\binom{n}{2}$ for this particular pattern.

\section{Generating Functions for Patterns in $S_3$}
\label{sec:calcs}
Let $f^p_n$ and $f^p(t)$ be as in \eqref{eqn:genfuncoeff} and \eqref{eqn:genfun} in Section \ref{sec:in}.
Then by Theorem \ref{mainthm} we have $f^{321}_n=\infty$ for all $n$.
However, for all of the other patterns $p\in S_3$ we can still compute $f^p(t)$.

\begin{thm}
\label{calcs1}
Let $f^p(t)$ be as above.
Then
\begin{align}
f^{123}(t)&=0,\\
f^{132}(t)=f^{213}(t)&=\sum_{n=2}^\infty t^n,\\
f^{231}(t)=f^{312}(t)&=\sum_{n=2}^\infty\binom{2n-1}{n}t^n.
\end{align}
\end{thm}

To make the proof easier, we first study a few operations on $\wt{S}_n$ that interact with pattern avoidance in a predictable way.

\begin{lemma}
\label{inversemap}
Let $\omega\in\wt{S}_n$ and $p\in S_m$.
Then $\omega$ avoids $p$ if and only if $\omega^{-1}$ avoids $p^{-1}$.
\end{lemma}

\begin{proof}
The proof is the same as the one for non-affine permutations given in \cite[Lemma 1.2.4]{West}.
Suppose $\omega$ contains $p$, so that $\omega_{i_1}\omega_{i_2}\cdots\omega_{i_m}$ is an occurrence of $p$ in $\omega$.
Let $j_k=\omega_{i_k}$ for $1\le k\le m$.
Then $\omega^{-1}_{j_1}\cdots\omega^{-1}_{j_m}$ will give an occurrence of $p^{-1}$ in $\omega^{-1}$.
\end{proof}

Now define a map $\sigma_r:\wt{S}_n\to\wt{S}_n$ by setting
\begin{displaymath}\sigma_r(\omega)_i=\begin{cases}\omega_{i-1}+1,&\text{ if }2\le i\le n,\\\omega_n-n+1,&\text{ if }i=1.\end{cases}\end{displaymath}
This has the effect of shifting the base window of $\omega$ one space to the right, while preserving the relative order of the elements.
The affine inversion table of $\sigma_r(\omega)$ is a barrel shift of the affine inversion table of $\omega$ one space to the right.
Similarly, define $\sigma_\ell=\sigma_r^{-1}$, which will perform a barrel shift one space to the left.
Thus $\sigma_r$ is the length-preserving automorphism of $\wt{S}_n$ of order $n$ obtained by rotating the Coxeter graph one space clockwise.

For example, if $\omega=[5,-4,6,3]\in\wt{S}_4$, which has affine inversion table $(4,0,3,1)$, then
$\sigma_r(\omega)=[0,6,-3,7]$, which has affine inversion table $(1,4,0,3)$.

\begin{lemma}
\label{shift}
Let $\omega\in\wt{S}_n$ and $p\in S_m$.
The following are equivalent.
\begin{enumerate}
\item $\omega$ avoids $p$.
\item $\sigma_r(\omega)$ avoids $p$.
\item $\sigma_\ell(\omega)$ avoids $p$.
\end{enumerate}
\end{lemma}

\begin{proof}
The relative order of elements in $\omega$ is unchanged after applying $\sigma_r$ or $\sigma_\ell$.
Hence if $\omega_{i_1}\cdots\omega_{i_m}$ is an occurrence of $p$ in $\omega$, then $\omega_{i_1+1}\cdots\omega_{i_m+1}$
is an occurrence of $p$ in $\sigma_r(\omega)$ and $\omega_{i_1-1}\cdots\omega_{i_m-1}$ is an occurrence
of $p$ in $\sigma_\ell(\omega)$.
\end{proof}

We are now ready to enumerate the affine permutations that avoid a given pattern in $S_3$.

\begin{proof}[Proof of Theorem \ref{calcs1}]
For any $\omega\in\wt{S}_n$, the entries $\omega_1\omega_{1+n}\omega_{1+2n}$ are always an occurrence of 123 in $\omega$.
Hence $f_n^{123}=0$ for all $n$.
If $\omega$ has a descent at $\omega_i$ so that $\omega_i>\omega_{i+1}$,
then there is some translate $i-sn$ such that $\omega_{i-sn}<\omega_{i+1}$.
Hence $\omega_{i-sn}\omega_i\omega_{i+1}$ is an occurrence of 132 in $\omega$.
Also, $\omega_{i+n}>\omega_{i+1}$ so that $\omega_i\omega_{i+1}\omega_{i+n}$ is an occurrence of 213 in $\omega$.
Thus the only affine permutation that can avoid 132 or 213 is the identity.
Hence $f_n^{132}=f_n^{213}=1$.

By Lemma \ref{inversemap} we have $f_n^{231}=f_n^{312}$.
Thus it remains to compute $f_n^{231}$.
So suppose $\omega$ avoids 231.
We first show $\omega$ is in a proper parabolic subgroup that depends on the position and value of the maximal element of the base window.

Let $\alpha$ be the index such that $\omega_\alpha=\max\{\omega_1,\dots,\omega_n\}$.
First suppose $\omega_\alpha>n+\alpha-1$.
Shift $\omega$ to the left $\alpha-1$ times, setting $\nu=\sigma_\ell^{\alpha-1}(\omega)$.
Then $\nu_1=\omega_\alpha-\alpha+1>n$.
Since $\nu$ must satisfy \eqref{eqn:affinesum}, there must exist some $1<j\le n$ with $\nu_j\le0$.
Then $\nu_{1-n}\nu_1\nu_j$ is an occurrence of 231 in $\nu$.
By Lemma \ref{shift}, $\omega$ contains 231, which is a contradiction.
So we must have $n\le\omega_\alpha\le n+\alpha-1$.

Now let $u=\sigma_\ell^{\omega_\alpha-n}(\omega)$.
Set $i=\alpha-\omega_\alpha+n$ so that $u_i=n$.
If $\{u_1,\dots,u_n\}\not=[n]$, then since $u$ must satisfy \eqref{eqn:affinesum}, there is some $1\le j,k\le n$ such that $u_j<0$ and $u_k>n$.
Since $\omega_\alpha$ was chosen to be maximal, we must have $i<k$.
Then $u_iu_ku_{j+n}$ will give an occurrence of 231 in $u$ and hence also in $\omega$ by Lemma \ref{shift}, giving a contradiction.
Hence $u\in S_n\subset\wt{S}_n$.

Let $C_n=\frac{1}{n+1}\binom{2n}{n}$ be the $n^{\text{th}}$ Catalan number.
Recall from \cite{art3} that there are $C_n$ 231-avoiding permutations in $S_n$.
Again, suppose $\omega_\alpha=\max\{\omega_1,\dots,\omega_n\}$ and $\omega_\alpha=n+\alpha-i$, for some $1\le i\le\alpha$.
Then $u=\sigma_\ell^{\omega_\alpha-n}(\omega)$ is an element in $S_n$ with $u_i=n$.
Furthermore, we have $u_h<u_j$ for every pair $h<i<j$.
There are $C_{i-1}C_{n-i}$ such permutations.
Summing over all possible values of $i$ gives
\begin{displaymath}\sum_{i=1}^\alpha C_{i-1}C_{n-i}=\sum_{i=0}^{\alpha-1}C_iC_{n-1-i}\end{displaymath}
many 231-avoiding affine permutations whose maximal value in the base window occurs at index $\alpha$.
Summing over all $1\le\alpha\le n$ then gives
\begin{equation}f_n^{231}\le\sum_{\alpha=1}^n\left(\sum_{i=0}^{\alpha-1}C_iC_{n-1-i}\right).\label{eqn:messy231}\end{equation}
Using the defining recurrence, \begin{equation}C_n=\sum_{i=0}^{n-1}C_iC_{n-1-i},\end{equation} for the Catalan numbers, \eqref{eqn:messy231} simplifies to
\begin{equation}f_n^{231}\le\frac{(n+1)}{2}C_n=\binom{2n-1}{n}.\label{eqn:231}\end{equation}

Conversely, if $u\in S_n\subset\wt{S}_n$ is a 231-avoiding permutation with $u_i=n$,
then $\sigma_r^{j}(u)$ will be a 231-avoiding affine permutation for any $0\le j\le n-i$.
Thus we actually have equality in \eqref{eqn:231}, completing the proof.
\end{proof}

\section{Generating Functions for Patterns in $S_4$}
\label{sec:calcs2}
We now look at pattern avoidance for patterns in $S_4$.
There are 24 patterns to consider, although for all but three patterns, $f^p(t)$ is easy to compute.
First let \begin{displaymath}P=\{1432,2431,3214,3241,3421,4132,4213,4231,4312,4321\}.\end{displaymath}
By Theorem \ref{mainthm}, if $p\in P$, then $f_n^p=\infty$, so $f^p(t)$ is not defined.

\begin{thm}
\label{calcs2}
We have
\begin{align}
f^{1234}(t)&=0,\\
f^{1243}(t)=f^{1324}(t)=f^{2134}(t)=f^{2143}(t)&=\sum_{n=2}^\infty t^n,\\
f^{1342}(t)=f^{1423}(t)=f^{2314}(t)=f^{3124}(t)&=\sum_{n=2}^\infty\binom{2n-1}{n}t^n.
\end{align}
\end{thm}

\begin{proof}
As in Theorem \ref{calcs1} there are no affine permutations avoiding 1234, and only the identity permutation avoids 1243, 1324, 2134, or 2143.
If $\omega_{i_1}\omega_{i_2}\omega_{i_3}$ is an occurrence of 231 in $\omega$,
then there is some translate $i_1-sn$ such that $\omega_{i_1-sn}<\omega_{i_3}$.
Hence $\omega_{i_1-sn}\omega_{i_1}\omega_{i_2}\omega_{i_3}$ is an occurrence of 1342 in $\omega$.
Conversely, if $\omega$ avoids 231, then it must also avoid any pattern containing 231, namely 1342. 
This shows $f^{1342}_n=f^{231}_n$.
Similarly, we also have $f^{1423}_n=f^{2314}_n=f^{3124}_n=f^{231}_n$.
\end{proof}

Based on some initial calculations, we also have the following conjectures for the remaining three patterns in $S_4$.

\begin{conj}
The following equalities hold.
\begin{align}
f^{3142}_n&=\sum_{k=0}^{n-1}\frac{(n-k)}{n}\binom{n-1+k}{k}2^k\label{eqn:conj1}\\
f^{3412}_n=f^{4123}_n&=\frac{1}{3}\sum_{k=0}^n\binom{n}{k}^2\binom{2k}{k}\label{eqn:conj2}
\end{align}
\end{conj}

Note that \eqref{eqn:conj1} is sequence A064062 and \eqref{eqn:conj2} is sequence A087457 in \cite{Sloanes}.
It is also worth comparing \eqref{eqn:conj2} to the number of 3412-avoiding, non-affine permutations given in \cite[\S7]{Ges:90} as
\begin{equation}
u_3(n)=2\sum_{k=0}^n\binom{n}{k}^2\binom{2k}{k}\frac{3k^2+2k+1-n-2kn}{(k+1)^2(k+2)(n-k+1)}.
\end{equation}

\section*{Acknowledgements}
\label{sec:ack}
The author would like to thank Brant Jones for inspiring him to ask this question
and Sara Billey for all of her guidance and helpful conversations in the construction of this paper.
He is also indebted to Neil Sloane, whose encyclopedia of integer sequences was invaluable for forming the conjectures in this paper.

\bibliographystyle{abbrvnat}

\label{sec:biblio}

\end{document}